\documentclass[a4paper,12pt]{article}

\usepackage[intlimits]{amsmath}
\usepackage{amssymb}
\usepackage{amsthm}
\usepackage{mathrsfs}
\usepackage{cite}
\usepackage{mathrsfs}
\usepackage{color}
\usepackage[colorlinks=true,linkcolor=blue]{hyperref}

\newcommand{\pr}{\partial}
\newcommand{\dom}{\Omega}
\newcommand{\sig}{\Sigma}
\newcommand{\gam}{\Gamma}

\newcommand{\hb}{H^{1/2}_{co}}
\newcommand{\hbi}{H^{-1/2}}
\newcommand{\s}{\mathcal{S}}
\newcommand{\ep}{\epsilon}

\newcommand{\el}{\mathcal{L}}

\newcommand{\R}{\mathbb{R}}

\newcommand{\dd}{\,\text{d}}

\newtheorem{thm}{Theorem}[section]

\newtheorem{lem}{Lemma}[section]
\newtheorem{cor}{Corollary}[section]
\newtheorem{dfn}{Definition}[section]

\title{Propagation of smallness for an elliptic PDE with piecewise Lipschitz coefficients}
\date{}
\author{C\u{a}t\u{a}lin I. C\^{a}rstea\thanks{School of Mathematics, Sichuan University, Chengdu, Sichuan, 610064, P.R.China; email: catalin.carstea@gmail.com} \and Jenn-Nan Wang\thanks{Institute of Applied Mathematical Sciences, NCTS, National Taiwan University, Taipei 106, Taiwan. email: jnwang@math.ntu.edu.tw}}

\begin{document}
\maketitle

\begin{abstract}
In this paper we derive a propagation of smallness result for a scalar second elliptic equation in divergence form whose leading order coefficients are Lipschitz continuous on two sides of a $C^2$ hypersurface that crosses the domain, but may have jumps across this hypersurface.  Our propagation of smallness result is in the most general form regarding the locations of domains, which may intersect the interface of discontinuity. At the end, we also list some consequences of the propagation of smallness result, including stability results for the associated Cauchy problem, a propagation of smallness result from sets of positive measure, and a quantitative Runge approximation property.
\end{abstract}

\section{Introduction}

Propagation of smallness is a quantitative form of the unique continuation property for solutions of partial differential equations. It can be regarded as a generalization of Hadamard three-circle theorem for analytic functions. For linear second order elliptic equations with nice coefficients, the propagation of smallness is well understood, see for example \cite{Brummelhuis1995} or the survey article \cite{ARRV} (and references therein). In this paper, we aim to study the propagation of smallness for second order elliptic equations with jump-type discontinuous leading order coefficients. 

The highlight of our result is that the domains in the propagation of smallness are arbitrarily chosen and may intersect the interface of discontinuity. It is also important to note that we obtain an inequality with exactly the same dependence on the geometry of the domains involved as in the classical result for Lipschitz leading order coefficients. This implies that a number of consequences of the classical result also apply to the type of piecewise Lipschitz leading order coefficients we are considering here. These consequences include stability results for the associated Cauchy problem, such as the ones proved in \cite{ARRV}, propagation of smallness from sets of positive measure, as the one proved in \cite{MV}, or the quantitative Runge approximation property  developed in \cite{RS}. Propagation of smallness and quantitative Runge approximation property have important applications to inverse problems such as the identification of obstacles by boundary measurements, or in the proof of stability results.

\subsection{Notations}

To better describe the main result, we would like to introduce several notations. Let $U\subset\R^n$, $n\geq 2$ be an open bounded domain. Suppose we have coefficients $a_{jk}, b_j, q\in L^\infty(U)$, $j,k=1,\ldots,n$. We will say that 
\begin{equation*}
\gamma:=\left((a_{jk})_{jk}, (b_j)_j, q  \right)\in \mathscr{V}(U,\lambda,M,K_1,K_2),
\end{equation*}
where $\lambda, M, K_1, K_2$ are positive constants, if
\begin{equation*}
\lambda |\xi|^2\leq \sum_{jk}a_{jk}(x)\xi_j\xi_k\leq\lambda^{-1}|\xi|^2,\quad\forall x\in U,
\end{equation*}
\begin{equation*}
||a_{jk}||_{C^{0,1}(U)}\leq M,\quad \|b_j\|_{L^\infty(U)}\leq K_2, \quad \|q\|_{L^{\infty}(U)}\leq K_1.
\end{equation*}
In the case when $b_j=0$, $j=1,\ldots, n$, we will write
\begin{equation*}
\left((a_{jk})_{jk},  q  \right)\in \mathscr{V}_0(U,\lambda,M,K_1),
\end{equation*}
With the set of coefficients $\gamma$, we define the second order elliptic operator $\el_\gamma$ that acts on a function $u$ as follows
\begin{equation*}
\el_\gamma u=\sum_{jk}\pr_j(a_{jk}\pr_k u)+\sum_j b_j\pr_j u+qu.
\end{equation*}

Suppose now that  $\Omega\subset\R^n$ is a Lipschitz domain and $\sig\subset\dom$ is a $C^2$ hypersurface. Further assume that $\dom\setminus\sig$ only has two connected components, which we denote $\dom_\pm$. If we have coefficients  $a^\pm_{jk}, b_j, q\in L^\infty(\dom_\pm)$, $j,k=1,\ldots,n$, such that
\begin{equation*}
\gamma^\pm:=\left((a^\pm_{jk})_{jk}, (b_j)_j, q  \right)\in \mathscr{V}(\dom_\pm,\lambda,M,K_1,K_2),
\end{equation*}
we will use the notation $\el_\gamma$ to denote the operator
\begin{equation*}
\el_\gamma u=\sum_{jk}\left[\chi_{\dom_+}\pr_j(a^+_{jk}\pr_k u)+\chi_{\dom_-}\pr_j(a^-_{jk}\pr_k u)\right]+\sum_j b_j\pr_j u+qu.
\end{equation*}

For $P\in\R^n$ and $r>0$, $B_r(P)$ will denote the open ball with center $P$ and radius $r$. For an open set $A\subset \R^n$ and a number $s>0$, we will use the notations
\begin{equation*}
A^s=\{x\in\R^n:\mbox{dist}(x,A)<s\},
\end{equation*}
\begin{equation*}
A_s=\{x\in A:\mbox{dist}(x,\pr A)<s\},
\end{equation*}
and
\begin{equation*}
sA=\{sx:x\in A\}.
\end{equation*}

\begin{dfn}\label{def-sigma}
We say that $\sig$ is $C^2$ with constants $r_0$, $K_0$ if for any point $P\in\sig$, after a rigid transformation, $P=0$ and 
\begin{equation*}
\dom_\pm\cap C_{r_0, K_0}(0)=\{(x,y): x\in \R^{n-1}, |x|<r_0, y\in\R, y\gtrless\psi(x)\},
\end{equation*}
where $\psi$ is a $C^2$ function such that $\psi(0)=0$, $\nabla_x\psi(0)=0$,  $\|\psi\|_{C^2(B_{r_0}(0))}\leq K_0$, and
\begin{equation*}
C_{r_0,K_0}(0)=\{(x,y):x\in \R^{n-1}, |x|<r_0, |y|\leq\frac{1}{2}K_0r_0^2\}.
\end{equation*}
\end{dfn}

If $\sig$ is as above, then we may "flatten'' the boundary around the point $P$ (without loss of generality $P=0$) via the local $C^2$-diffeomeorphism
\begin{equation*}
\Psi_P(x,y)= (x, y-\psi(x)).
\end{equation*}

\subsection{Main result and outline}

Let $D\subset\subset\dom$ be open and connected. 
Suppose that $\sig\subset\dom$ is a $C^2$ hypersurface with constants $r_0$ and $K_0$. Further assume that $\dom\setminus\sig$ has two connected components, $\dom_\pm$.
Let $a^\pm_{jk}, q\in L^\infty(\dom_\pm)$, $j,k=1,\ldots,n$, be coefficients such that
\begin{equation*}
\left((a^\pm_{jk})_{jk},  q  \right)\in \mathscr{V}_0(\dom_\pm,\lambda,M,K_1).
\end{equation*}

With these assumptions, we will prove a propagation of smallness result as follows.

\begin{thm}\label{main-thm-1}
Suppose $u\in H^1(\dom)$ solves
\begin{equation*}
\el_\gamma u=f+\nabla \cdot F,\quad \|f\|_{L^2(\dom)}+\|F\|_{L^2(\dom)}\leq\epsilon.
\end{equation*}
There exist a constant $ h_0>0$, depending on $\lambda$, $M$, $K_1$, $r_0$, $K_0$, $\sig$, such that if  $0<h<h_0$,  
 $ r/2>h$, $D\subset\dom$ is connected, open, and such that $B_{r}(x_0)\subset D$, $\mbox{\rm dist} (D,\pr\dom)\geq h$, then 
\begin{equation*}
\|u\|_{L^2(D)}\leq C (\|u\|_{L^2(B_r(x_0))}+\epsilon)^\delta(\|u\|_{L^2(\dom)}+\epsilon)^{1-\delta},
\end{equation*}
where
\begin{equation*}
C=C_1\left(\frac{|\dom|}{h^{n}}\right)^\frac{1}{2},\quad \delta\geq\tau^{\frac{C_2|\dom|}{h^{n}}},
\end{equation*}
with $C_1, C_2>0$, $\tau\in(0,1)$ depending on $\lambda$, $M$, $K_1$,  $r_0$, $K_0$.
\end{thm}

We want to point out that the propagation of smallness we obtained is in the most general form regarding the locations of $D$ and $B_{r_0}(x_0)$, which may intersect the interface $\Sigma$. The strategy of proving Theorem~\ref{main-thm-1} consists two parts.  When we are at one side of the interface, we can use the usual propagation of smallness for equations with Lipschitz coefficients. When we near the interface, we then use the three-region inequality derived in \cite{FLVW}. The three-region inequality of \cite{FLVW} is used to propagate the smallness across the interface. 

The rest of the paper is organized as follows. In section \ref{KNOWN} we  recall two known results. The first is a propagation of smallness result \cite[Theorem 5.1]{ARRV} analogous to our own, in the case of Lipschitz leading order coefficients. The second is a ``three-region inequality''  \cite[Theorem 3.1]{FLVW} for leading order coefficients which are Lipschitz except across a $C^2$ hypersurface. The rest of the section is concerned with extending the three regions inequality to a slightly richer family of regions.

In section \ref{PROPAGATION} we use the three-region inequality we have established in section \ref{KNOWN} to prove a propagation of smallness result with somewhat worse constants than the ones in Theorem \ref{main-thm-1}. Then in section \ref{PROOF} we use this intermediate propagation of smallness result to prove Theorem \ref{main-thm-1}.

Finally, in section \ref{CONSEQUENCES}, following \cite{ARRV}, \cite{MV}, and \cite{RS}, we list  a few consequences of our main result. These are given without  proofs, as these would be identical to the ones given in \cite{ARRV}, \cite{MV}, and \cite{RS}.

\section{Known results and extensions}\label{KNOWN}

In this section we  recall the propagation of smallness result for Lipschitz leading order coefficients established in \cite{ARRV} and the three-region inequality proved in \cite{FLVW} for leading order coefficients that are Lipschitz except on a plane that intersects the domain. We then state and prove an extension of the three-region inequality which will introduce a scaling parameter to the family of regions for which the inequality applies. 

\subsection{
Lipschitz leading order coefficients}

Assume that
$
\left((a_{jk})_{jk},  q  \right)\in \mathscr{V}_0(U,\lambda,M,K_1).
$
We then have the following.

\begin{thm}[{\cite[Theorem 5.1]{ARRV}}]\label{propagation}
Let $u\in H^1(U)$ be a weak solution to the elliptic equation
\begin{equation*}
\el_\gamma u=f+\nabla \cdot F,\quad \|f\|_{L^2(U)}+\|F\|_{L^2(U)}\leq\epsilon.
\end{equation*}
Let $0<h< r/2$, $D\subset U$ connected, open, and such that $B_{\frac{r}{2}}(x_0)\subset D$, $\mbox{\rm dist} (D,\pr U)\geq h$. Then
\begin{equation*}
\|u\|_{L^2(D)}\leq C (\|u\|_{L^2(B_r(x_0))}+\epsilon)^\delta(\|u\|_{L^2(U)}+\epsilon)^{1-\delta},
\end{equation*}
where
\begin{equation*}
C=C_1\left(\frac{|U|}{h^n}\right)^{\frac{1}{2}},\quad \delta\geq\tau^{\frac{C_2|U|}{h^n}},
\end{equation*}
with $C_1, C_2>0$, $\tau\in(0,1)$ depending on $\lambda$, $M$, $K_1$.
\end{thm}

\subsection{
Piecewise Lipschitz leading order coefficients}

Let
\begin{equation*}
\R_+^n=\left\{(x,y)\in\R^{n-1}\times\R:y>0\right\},
\end{equation*}
\begin{equation*}
\R_-^n=\left\{(x,y)\in\R^{n-1}\times\R:y<0\right\},
\end{equation*}
and assume that
\begin{equation*}
\tilde \gamma^\pm=\left((\tilde a^\pm_{jk})_{jk},(\tilde b_j)_j,  \tilde q  \right)\in \mathscr{V}(\R^n_\pm,\lambda,M,K_1).
\end{equation*}

\begin{thm}[{\cite[Theorem 3.1]{FLVW}}]\label{thm-FLVW}
There exist $\alpha_\pm$, $\delta_0$, $\tau_0$, $\beta$, $C$, $R$ positive constants depending on $\lambda$, $M$, $K_1$, $K_2$, such that if $0<\delta<\delta_0$, $0<R_1, R_2\leq R$, and 
\begin{equation*}
\el_{\tilde\gamma} u=0,\text{ in } U_3,
\end{equation*}
 then
\begin{equation*}
\int_{U_2}|u|^2\leq (e^{\tau_0 R_2}+CR_1^{-4})
\left(\int_{U_1} |u|^2\right)^{\frac{R_2}{2R_1+3R_2}}
\left(\int_{U_3} |u|^2\right)^{\frac{2R_1+2R_2}{2R_1+3R_2}},
\end{equation*}
where
\begin{equation*}
U_1=\{ z\geq-4R_2,\frac{R_1}{8a}<y<\frac{R_1}{a}\},
\end{equation*}
\begin{equation*}
U_2=\{-R_2\leq z\leq \frac{R_1}{2a}, y<\frac{R_1}{8a}\},
\end{equation*}
\begin{equation*}
U_3=\{z\geq-4R_2, y<\frac{R_1}{a}\},
\end{equation*}
\begin{equation*}
a=\frac{\alpha_+}{\delta},\quad 
z(x,y)=\frac{\alpha_-y}{\delta}+\frac{\beta y^2}{2\delta^2}-\frac{|x|^2}{2\delta}.
\end{equation*}
\end{thm}

Note that in \cite{FLVW}, the function $u$ is required to be a solution in $\R^n$. It is however clear from their proof that it only needs to solve the equation in $U_3$. The proof of the three-region inequality is based on the Carleman estimate derived in \cite{CFLVW17} (or \cite{RR10}).

\subsection{Scaling the three regions}

When trying to prove a propagation of smallness result, the family of regions given in Theorem \ref{thm-FLVW} has one important drawback, namely that if we choose the parameters $R_1=\theta \bar R_1$, $R_2=\theta\bar R_2$, $\theta\in(0,1)$, the vertical (i.e. $y$-direction) size of the regions would scale like $\theta$, while their horizontal (i.e. $x$-direction) size would scale like $\theta^{\frac{1}{2}}$. Using just these two parameters in the proof would then lead to constants in the propagation of smallness inequality (i.e. the constants $C$ and $\delta$ in Theorem \ref{main-thm-1}) that depend on the geometry of $\dom$, $D$, and $B_{r}(x_0)$ in a way that is not invariant under a rescaling of these sets. 

In order to derive a propagation of smallness result that is more closely analogous to \cite[Theorem 5.1]{ARRV}, we need to introduce another parameter to the family of three regions. 

Assume that
\begin{equation*}
\tilde \gamma^\pm=\left((\tilde a^\pm_{jk})_{jk},(\tilde b_j)_j,  \tilde q  \right)\in \mathscr{V}(\R^n_\pm,\lambda,M,K_1, K_2).
\end{equation*}
For $0<\theta\leq 1$, let
\begin{multline*}
\mathcal{L}^\theta_{\tilde\gamma} v=\sum_{jk}\left[\chi_{R^n_+}\pr_j(\tilde a^+_{jk}(\theta\cdot)\pr_k v)+\chi_{R^n_-}\pr_j(\tilde a^-_{jk}(\theta\cdot)\pr_k v)\right]\\[5pt]
+\sum_j \theta \tilde b_j(\theta\cdot)\pr_j v+\theta^2 \tilde q(\theta\cdot)v.
\end{multline*}
Note that if $\left( (\tilde a^\pm_{jk})_{j,k=1}^n,(\tilde b_j)_{j=1}^n,\tilde q)  \right)\in\mathscr{V}(\R^{n}_\pm, \lambda,M,K_1,K_2)$, then
\begin{equation*}
\left( (\tilde a^\pm_{jk}(\theta\cdot))_{j,k=1}^n,(\theta \tilde b_j(\theta\cdot))_{j=1}^n,\theta^2 \tilde q(\theta\cdot)  \right)\in\mathscr{V}(\R^{n}_\pm, \lambda,\theta M,\theta^2 K_1, \theta K_2).
\end{equation*}
Suppose $\el_{\tilde\gamma}=0$ in $\dom$, and let
\begin{equation*}
u^\theta(x)=\theta^{-2}u(\theta x).
\end{equation*}
Then
\begin{equation*}
\mathcal{L}^\theta_{\tilde\gamma} u^\theta=0,\text{ in }\theta^{-1}\dom. 
\end{equation*}

If $U\subset \theta^{-1}\dom$, note that
\begin{equation*}
\int_{\theta U}|u(x)|^2\dd x=\theta^{n+4}\int_U |u^\theta(y)|^2\dd y.
\end{equation*}
We therefore obtain, by scaling, the following corollary to Theorem \ref{thm-FLVW}.
\begin{thm}\label{thm-FLVW-2}
If $0<\delta<\delta_0$, $0<R_1, R_2\leq R$, $\theta\in (0,1]$, and 
\begin{equation*}
\el_{\tilde\gamma} u=0,\text{ in } \theta U_3,
\end{equation*}
 then 
\begin{equation*}
\int_{\theta U_2}|u|^2\leq (e^{\tau_0 R_2}+CR_1^{-4})
\left(\int_{\theta U_1} |u|^2\right)^{\frac{R_2}{2R_1+3R_2}}
\left(\int_{\theta U_3} |u|^2\right)^{\frac{2R_1+2R_2}{2R_1+3R_2}}.
\end{equation*}
\end{thm}

For $u$ a solution to an inhomogeneous equation, we  easily have a similar result.
\begin{cor}\label{cor-FLVW-2}
Under the assumptions above, if 
\begin{equation*}
\el_{\tilde\gamma} u=f+\nabla \cdot F,\text{ in }\theta U_3,
\end{equation*}
\begin{equation*}
\|f\|_{L^2(\dom)}+\|F\|_{L^2(\dom)}\leq\epsilon,
\end{equation*}
then  
\begin{multline*}
\|u\|_{L^2(\theta U_2)} \leq C (e^{\tau_0 R_2}+CR_1^{-4})^{\frac{1}{2}}\\[5pt] 
\times \left( \|u\|_{L^2(\theta U_1)}+\epsilon\right)^{\frac{R_2}{2R_1+3R_2}}  \left(\|u\|_{L^2(\theta U_3)}+\epsilon\right)^{\frac{2R_1+2R_2}{2R_1+3R_2}}.
\end{multline*}
\end{cor}

\begin{proof}
Let $u_0\in H^1_0(\theta U_3)$ be the solution to 
\begin{equation*}
\el_{\tilde\gamma} u_0=f+\nabla \cdot F.
\end{equation*}
Then 
\begin{equation*}
\|u_0\|_{L^2(\theta U_3)}\leq C(\|f\|_{L^2(\dom)}+\|F\|_{L^2(\dom)})\leq C\epsilon.
\end{equation*}
Since
\begin{equation*}
\el(u-u_0)=0
\end{equation*}
we can apply Theorem \ref{thm-FLVW-2} to $u-u_0$ and the claim follows immediately.
\end{proof}

\section{An intermediate propagation of smallness result}\label{PROPAGATION}

In this section we  assume that  $D\subset\subset\dom$ is  open and connected,  that $\sig\subset\dom$ is a $C^2$ hypersurface with constants $r_0$ and $K_0$, and  that $\dom\setminus\sig$ and $D\setminus{\sig}$ both have two connected components each, denoted by $\dom_\pm$ and $D_\pm$ respectively. We will consider coefficients
\begin{equation*}
\gamma:=\left((a^\pm_{jk})_{jk},  q  \right)\in \mathscr{V}_0(\dom_\pm,\lambda,M,K_1).
\end{equation*}
We can now prove the following propagation of smallness result. 

\begin{thm}\label{propagation-thm}
Suppose $u\in H^1(\dom)$ solves
\begin{equation*}
\el_{\gamma} u=f+\nabla \cdot F,\quad \|f\|_{L^2(\dom)}+\|F\|_{L^2(\dom)}\leq\epsilon.
\end{equation*}
Then there exist $h_0>0$ depending on $\lambda$, $M$, $K_1$, $r_0$, $K_0$, $\sig$, such that if $0<h<h_0$, $h< r/2$,  $B_{r}(x_0)\subset D_+$, and $\mbox{\rm dist} (D,\pr\dom)\geq h$, then 
\begin{equation*}
\|u\|_{L^2(D)}\leq C (\|u\|_{L^2(B_r(x_0))}+\epsilon)^\delta(\|u\|_{L^2(\dom)}+\epsilon)^{1-\delta},
\end{equation*}
where
\begin{equation*}
C=C_1\left(\frac{|\dom|}{h^{n}}\right)\left[1+\left(\frac{|\sig\cap \dom|}{h^{n-1}} \right)^\frac{1}{2}\right],\quad \delta\geq\tau^{\frac{C_2|\dom|}{h^{n}}},
\end{equation*}
with $C_1, C_2>0$, $\tau\in(0,1)$ depending on $\lambda$, $M$, $K_1$, $r_0$, $K_0$.
\end{thm}

The difficult part of the proof is obtaining $L^2$ estimates of the solution in  a neighborhood of $\sig$. We will use Corollary \ref{cor-FLVW-2} above in order to accomplish this. In order to adapt that result to the possibly curved surface $\sig$, we need to first consider how the three regions transform under the local boundary straightening diffeomorphisms $\Psi_P$.

\subsection{Preimages of the three regions}

Pick a point $P\in\sig$ and set $P=0$ without loss of generality. Let  $(x,y)\in C_{r_0,K_0}(0)$. We will try to determine when $(x,y)\in\Psi^{-1}_P(\theta U_2)$. To this end, we introduce the notation
\begin{equation*}
x'=x,\quad y'=y-\psi(x).
\end{equation*}
It is clear that $(x,y)\in\Psi^{-1}_P(\theta U_2)$ if and only if $\theta^{-1}(x',y')\in U_2$. Because we expect the condition on the size of $(x,y)$ to be approximately of order $\theta$, we also introduce
\begin{equation*}
x''=\frac{x}{\theta},\quad y''=\frac{y}{\theta},
\end{equation*}
and expect to obtain a condition of order $1$ on these.
Finally, we introduce the function
\begin{equation*}
\zeta(x)=\frac{\psi(x)}{|x|^2},
\end{equation*}
which is bounded by our assumption on the regularity of $\sig$.
Then
\begin{multline*}
z(\theta^{-1}x',\theta^{-1}y')\\
=\frac{\alpha_-}{\delta}y''+\frac{\beta}{2\delta^2}y''^2-\frac{|x''|^2}{2\delta}-\frac{\alpha_-}{\delta}\frac{\psi(x)}{\theta}-\frac{\beta}{\delta^2}y''\frac{\psi(x)}{\theta}+\frac{\beta}{2\delta^2}\frac{\psi(x)^2}{\theta^2}.
\end{multline*}

Suppose $|x|^2+y^2=r^2$. Let $r''=\theta^{-1}r$, then
\begin{multline*}
z(\theta^{-1}x',\theta^{-1}y')=\frac{\alpha_-}{\delta}y''+\frac{1}{2\delta^2}(\delta+\beta)(y'')^2-\frac{(r'')^2}{2\delta}\\
-\frac{\alpha_-}{\delta}\theta\zeta(x)|x''|^2-\frac{\beta}{\delta^2}\theta\zeta(x)y''|x''|^2+\frac{\beta}{2\delta^2}\theta^2\zeta(x)^2|x''|^4
\end{multline*}
When $r''<r_1=\frac{\alpha_-\delta}{\delta+\beta}$, the minimum and maximum values of the fist three terms on the right hand side combined will be attained when $y''=\pm r''$ (endpoints of $[-r'',r'']$). Let $\|\zeta\|=\|\zeta\|_{L^\infty(B_{\R^{n-1}}(0,r_0))}$, then
\begin{equation*}
z(\theta^{-1}x',\theta^{-1}y')\leq \frac{\alpha_-}{\delta}(1+\theta\|\zeta\|r'')r''+\frac{\beta}{2\delta^2}(1+2\theta\|\zeta\|r''+\theta^2\|\zeta\|^2r''^2)r''^2,
\end{equation*}
\begin{equation*}
z(\theta^{-1}x',\theta^{-1}y')\geq -\frac{\alpha_-}{\delta}(1+\theta\|\zeta\|r'')r''+\frac{\beta}{2\delta^2}(1-2\theta\|\zeta\|r''-\theta^2\|\zeta\|^2r''^2)r''^2.
\end{equation*}
Suppose now that $r''<r_2$, where $r_2$ is chosen so that 
\begin{equation*}
2\|\zeta\|r_2+\|\zeta\|^2r_2^2<\frac{1}{2},
\end{equation*}
(which implies $r_2<1/(4\|\zeta\|)$). We have 
\begin{equation*}
z(\theta^{-1}x',\theta^{-1}y')\leq \frac{3\alpha_-}{2\delta}r''+\frac{3\beta}{4\delta^2}r''^2,
\end{equation*}
\begin{equation*}
z(\theta^{-1}x',\theta^{-1}y')\geq -\frac{3\alpha_-}{2\delta}r''+\frac{\beta}{4\delta^2}r''^2.
\end{equation*}
Incidentally, note that if $r''<r_2$, then
\begin{equation*}
\theta^{-1}y'=y''-\theta^{-1}\psi(x)<r''+\theta\|\zeta\|r''^2<\frac{3}{2}r'',
\end{equation*}
so the condition $\theta^{-1}y'<\frac{R_1}{8a}$ is satisfied if $r''<\frac{R_1}{12 a}$. Hence, $\Psi_P(B(0,r))\subset \theta U_2$ if
\begin{equation*}
\frac{3\alpha_-}{2\delta}r''+\frac{3\beta}{4\delta^2}r''^2<\frac{R_1}{2a}
\end{equation*}
and
\begin{equation*}
\frac{3\alpha_-}{2\delta}r''-\frac{\beta}{4\delta^2}r''^2<R_2.
\end{equation*}
Consequently, if we only consider $r''<r_3=\frac{2\alpha_-\delta}{\beta}$, then we have that $\Psi_P(B(0,r))\subset \theta U_2$ if 
\begin{equation*}
\frac{3\alpha_-}{\delta}r''<\frac{R_1}{2a}
\end{equation*}
and
\begin{equation*}
\frac{3\alpha_-}{2\delta}r''<R_2
\end{equation*}
In other words, we have proved the following lemma.
\begin{lem}\label{lem-U2}
If 
\begin{equation*}
r<\theta\min\left\{\frac{\delta R_1}{6a\alpha_-},\frac{2\delta R_2}{3\alpha_-}, \frac{R_1}{12a},\theta^{-1}r_0,r_1,r_2,r_3\right\},
\end{equation*}
then $\Psi_P(B(P,r))\subset \theta U_2$.
\end{lem}

Using the same notation as above, by simple estimates we get that if $(x',y')\in \theta U_3$, then
\begin{equation*}
\frac{|x'|^2}{\theta^2}<\frac{2\alpha_-}{a}R_1+\frac{\beta}{a^2\delta}R_1^2+8\delta R_2=:R_0^2,
\end{equation*}
\begin{equation*}
-\frac{\delta}{\beta}\left(\alpha_--\sqrt{\alpha_-^2 -8\beta R_2}\right)\leq \frac{y'}{\theta} \leq\frac{R_1}{a}.
\end{equation*}
Noting that 
\begin{equation*}
|y|^2=|y'+\psi(x)|^2\leq 2|y'|^2+2\|\zeta\|^2|x|^2,
\end{equation*}
we can show that 
\begin{lem}\label{lem-U3}
$\Psi_P^{-1}(\theta U_3)$ is contained in a ball of radius
{\scriptsize
\begin{equation*}
\theta\left[(1+2\|\zeta\|^2)\left(\frac{2\alpha_-}{a}R_1+8\delta R_2\right)+
\frac{1}{a^2}\left[2+(1+2\|\zeta\|^2)\frac{\beta}{\delta}   \right]R_1^2 +
\frac{128\delta^2R_2^2}{\left[\alpha_-+\sqrt{\alpha_-^2-8\beta  R_2}  \right]^2}      \right]^{1/2}
\end{equation*}
}
centered at $P$.
\end{lem}

Finally, we would like to estimate the distance between $\Psi_P^{-1}(\theta U_1)$ and $\sig\cap C_{r_0,K_0}$. Note that  $\mbox{dist} (\theta U_1, \pr\R^n_+)=\theta \frac{R_1}{8 a}$.  Recall that $R_0$ is such that $\Psi_P^{-1}(\theta U_3)$ intersects the plane $\{y=0\}$ in a set contained in a ball of radius $\theta R_0$ centered at $P$. Let $x_1, x_2\in B_{\R^{n-1}}(0,R_0)$, $Y>0$, and set
\begin{equation*}
d=\mbox{\rm dist}\left((x_1,\psi(x_1)),(x_2,\psi(x_2)+Y)\right).
\end{equation*}
Then
\begin{multline*}
d^2=|x_2-x_1|^2+|Y+\psi(x_2)-\psi(x_1)|^2\\
\geq |x_2-x_1|^2+|Y|^2+|\psi(x_2)-\psi(x_1)|^2-2|Y|\,|\psi(x_2)-\psi(x_1)|\\
\geq |x_2-x_1|^2+\frac{1}{2}|Y|^2-|\psi(x_2)-\psi(x_1)|^2.
\end{multline*}
We can estimate 
\begin{multline*}
|\psi(x_2)-\psi(x_1)|\leq |x_2-x_1|\int_0^1|\nabla \psi(x_1+t(x_2-x_1))|\dd t\\
\leq |x_2-x_1|K_0\int_0^1|x_1+t(x_2-x_1)|\dd t\leq K_0 R_0|x_2-x_1|,
\end{multline*}
therefore
\begin{equation*}
d \geq \frac{1}{2}|Y|^2+(1-K_0^2R_0^2)|x_2-x_1|^2.
\end{equation*}
If necessary, $R$ (given in Theorem~\ref{thm-FLVW}) can be changed so that $K_0^2R_0^2<1$. The above estimate, with $Y=\theta\frac{R_1}{8 a}$,  implies
\begin{lem}\label{lem-U1}
\begin{equation*}
\mbox{\rm dist}(\Psi_P^{-1}(\theta U_1),\Sigma)>\theta\frac{R_1}{16a}.
\end{equation*}
\end{lem}

\subsection{Proof of Theorem \ref{propagation-thm}}

Without loss of generality,  we may take $D$ to be the set 
\begin{equation*}
D=\{x\in\dom:\mbox{\rm dist}(x,\pr\dom)>h\}.
\end{equation*}
We pick $R_1$, $R_2$ so that we can apply Corollary \ref{cor-FLVW-2} at any point $P\in \sig\cap D$. By Lemma \ref{lem-U3}, there is a constant $d>0$, independent of $P$, such that $\Psi_P^{-1}(\theta U_3)\subset B_{\theta d}(P)$. We will choose $\theta$ such that $\theta d=\frac{h}{2}$, which implies $\Psi_P^{-1}(\theta U_3)\subset \dom$ for any  $P\in \sig\cap D$. Of course, this choice is not possible if $h$ is too large, so here we need to set $h_0$ low enough, depending on $r_0, K_0, \lambda, M, K_1, \sig$.

With this choice of parameters, by Lemma \ref{lem-U1}, there is a constant $1>\mu>0$, also independent on $P$, so that 
\begin{equation*}
\mbox{\rm dist}(\Psi_P^{-1}(\theta U_1), \sig)>\mu h.
\end{equation*}
 Note that, depending on the geometry of $\sig$, we again need to set $h_0$ and $R$ small enough so that $\Psi_P^{-1}(\theta U_1)\cap \sig^{\mu h}=\emptyset$, for any  $P\in \sig\cap D$.
 
By Lemma \ref{lem-U2}, there exists a constant $\nu>0$, and without loss of generality $\nu<\mu<1$, so that $B_{5\nu h}(P)\subset \Psi_P^{-1}(\theta U_2)$. By Vitali's covering lemma, there exist finitely many $P_1,\ldots,P_N\in \sig\cap D$ so that
\begin{equation}\label{qq}
\sig^{\nu h}\cap D\subset\bigcup_{j=1}^N\Psi_{P_j}^{-1}(\theta U_2),
\end{equation}
and the balls $B_{\nu h}(P_j)$ are pairwise disjoint. By this last property, since for small $h$ we have $|\sig^{\nu h}\cap D|\sim \nu h |\sig\cap D|$, it follows that there is a constant $C$ such that
\begin{equation}\label{qqq}
N\leq C\frac{|\sig\cap D|}{h^{n-1}}\leq C\frac{|\sig\cap\dom|}{h^{n-1}}.
\end{equation}

Let us denote $\tilde D= (D_+)^{h/2}\setminus(\sig^{\nu h}\cup \dom_-)$, then by Theorem \ref{propagation}, we have that
\begin{equation}\label{D-plus-estimate}
\|u\|_{L^2(\tilde D)}\leq C_+ (\|u\|_{L^2(B_r(x_0))}+\epsilon)^{\delta_+}(\|u\|_{L^2(\dom)}+\epsilon)^{1-\delta_+},
\end{equation}
where
\begin{equation*}
B_r(x_0)\subset\tilde D,\quad C_+=C_1\left(\frac{|\dom|}{h^n}\right)^{\frac{1}{2}},\quad \delta_+\geq\tau^{\frac{C_2|\dom|}{h^n}}.
\end{equation*}
The function $v=u\circ \Psi_{P_j}^{-1}\in H^1(\theta U_3)$ satisfies in $\theta U_3$ an equation of the form
\begin{equation*}
\el_{\tilde \gamma} v=\tilde f+\nabla \cdot \tilde F,\quad \|\tilde f\|_{L^2(\dom)}+\|\tilde F\|_{L^2(\dom)}\leq C\epsilon,
\end{equation*}
with the coefficients of the operator $\el_{\tilde \gamma}$ satisfying
\begin{equation*}
\tilde \gamma^\pm=\left((\tilde a^\pm_{jk})_{jk},(\tilde b_j)_j,  \tilde q  \right)\in \mathscr{V}(\R^n_\pm,\tilde \lambda,\tilde M,\tilde K_1, \tilde K_2),
\end{equation*}
with $C>0$ and the parameters $\tilde \lambda,\tilde M,\tilde K_1, \tilde K_2$ depending on $\lambda, M, K_1, K_2, r_0, K_0$. We can then pull back the three regions inequality of Corrolary \ref{cor-FLVW-2} and apply it to $u$ and the regions $\Psi_{P_j}^{-1}(\theta U_1)$, $\Psi_{P_j}^{-1}(\theta U_2)$, $\Psi_{P_j}^{-1}(\theta U_3)$.

Since 
\begin{equation*}
\Psi_{P_j}^{-1}(\theta U_1)\subset (D_+)^{h/2}\setminus(\sig^{\nu h}\cup \dom_-),
\end{equation*}
we have that
\begin{equation*}
\|u\|_{L^2(\Psi_{P_j}^{-1}(\theta U_2))}\leq C (\|u\|_{L^2((D_+)^{h/2}\setminus(\sig^{\nu h}\cup \dom_-))}+\epsilon)^\xi 
(\|u\|_{L^2(\dom)}+\epsilon)^{1-\xi},
\end{equation*}
where $\xi=\frac{R_2}{2 R_1+3 R_2}$. Combining this and \eqref{D-plus-estimate}, we obtain
\begin{equation*}
\|u\|_{L^2(\Psi_{P_j}^{-1}(\theta U_2))}\leq C_1'\left(\frac{|\dom|}{h^n}\right)^{\frac{\xi}{2}}
 (\|u\|_{L^2(B_r(x_0))}+\epsilon)^{\xi\delta_+}(\|u\|_{L^2(\dom)}+\epsilon)^{1-\xi\delta_+}.
\end{equation*}
Then it follows from \eqref{qq} and \eqref{qqq} that 
\begin{multline}\label{Sigma-estimate}
\|u\|_{L^2(\sig^{\nu h}\cap D)}\leq C_1''\left(\frac{|\sig\cap \Omega|}{h^{n-1}}\right)^{\frac{1}{2}}\left(\frac{|\dom|}{h^n}\right)^{\frac{\xi}{2}}\\[5pt]
\times (\|u\|_{L^2(B_r(x_0))}+\epsilon)^{\xi\delta_+}(\|u\|_{L^2(\dom)}+\epsilon)^{1-\xi\delta_+}.
\end{multline}

Applying Theorem \ref{propagation} again (now with an appropriate small ball $\tilde B_{\tilde r}\subset\Sigma^{\nu h}\cap D_-\subset\Sigma^{\nu h}\cap D$), we have
\begin{multline}\label{D-minus-estimate}
\|u\|_{L^2(D_-\setminus\sig^{\nu h})}\leq C_1'''\left(\frac{|\dom|}{h^n}  \right)^{\frac{1}{2}}\left(\frac{|\sig\cap \dom|}{h^{n-1}}\right)^{\frac{\delta_-}{2}}\left(\frac{|\dom|}{h^n}\right)^{\frac{\delta_-\xi}{2}}\\[5pt]
\times (\|u\|_{L^2(B_r(x_0))}+\epsilon)^{\delta_-\xi\delta_+}(\|u\|_{L^2(\dom)}+\epsilon)^{1-\delta_-\xi\delta_+},
\end{multline}
where
\begin{equation*}
\delta_-\geq\tau^{\frac{C_2'|\dom|}{h^n}}.
\end{equation*}
Combining the estimates \eqref{D-plus-estimate}, \eqref{Sigma-estimate}, and \eqref{D-minus-estimate}, we obtain the conclusion of Theorem \ref{propagation-thm}.

\section{The proof of Theorem \ref{main-thm-1}}\label{PROOF}

In this section we  will prove the main theorem of this paper. We begin with deriving a three balls inequality, which is a direct consequence of Theorem \ref{propagation-thm}. We would like to remark that a version of three balls inequality for the second order elliptic equation with jump-type discontinuous coefficients was obtained in \cite{DCR19}. However, the estimate  in \cite{DCR19} does not fit what we need. So we derive our own three balls inequality here to serve a building block in the proof of the main theorem. 

\subsection{Three balls inequality}

Here we  assume $\dom\subset\R^n$ is an open Lipschtiz domain,  $\sig$ is a $C^2$ hypersurface with constants $r_0$, $K_0$,  and $\dom\setminus \sig$ has two connected components, $\dom_\pm$. We also assume we have coefficients
\begin{equation*}
\left((a^\pm_{jk})_{jk},  q  \right)\in \mathscr{V}_0(\dom_\pm,\lambda,M,K_1).
\end{equation*}
With these assumptions, let $u\in H^1(\dom)$ be a solution of
\begin{equation*}
\el_{\gamma} u=f+\nabla\cdot F,\quad \|f\|_{L^2(\dom)}+\|F\|_{L^2(\dom)}\leq \epsilon.
\end{equation*}

\begin{thm}\label{three-balls}
There exist values $\bar r>0$, depending on $r_0$, $K_0$, such that if    $0<r_1<r_2<r_3<\bar r$, $Q\in\dom$, $\text{\rm dist}(Q,\pr\dom)>r_3$, then there exist $C>0$, $0<\delta<1$ such that
\begin{equation}\label{3ball}
\|u\|_{L^2(B_{r_2}(Q))}\leq C(\|u\|_{L^2(B_{r_1}(Q))}+\epsilon)^\delta (\|u\|_{L^2(B_{r_3}(Q))}+\epsilon)^{1-\delta}.
\end{equation}
 $C$, and $\delta$ depend on  $\lambda$, $M$, $r_0$, $K_0$, $K_1$, $\frac{r_1}{r_2}$, $\frac{r_2}{r_3}$, $\text{\rm diam}(\dom)$, $|\sig\cap\dom|$.
\end{thm}

\begin{proof}

We would like to use the propagation of smallness result with $r=\frac{r_1}{10}$, $D=B_{r_2}(Q)$, and $\dom=B_{r_3}(Q)$. We can choose the constant $\bar r$ so  that $B_{r_j}(Q)\setminus\sig$ can all only have at most two connected components. This would be the case for example if $\bar r\leq \min(r_0,\frac{1}{2}K_0r_0^2)$. Fix an $\bar r$ as described. Then we can always find $Q'\in B_{r_1}(Q)$ so that $B_{r_1/10}(Q')\subset B_{r_1}(Q)\cap\dom_+$ or $B_{r_1/10}(Q')\subset B_{r_1}(Q)\cap\dom_-$. Without loss of generality we may assume that $B_{r_1/10}(Q')\subset B_{r_1}(Q)\cap\dom_+$.

Let $g^\sig$ be the metric induced on $\sig$ by the Euclidean metric of $\R^n$. Around a point $P\in\sig$ at which we have chosen coordinates as in Definition \ref{def-sigma}, we can use the coordinates $(x_1, \ldots, x_{n-1})$ as a local map for $\sig$. In these coordinates
\begin{equation*}
g^\sig_{jk}=\delta_{jk}+\pr_j\psi\pr_k\psi.
\end{equation*}
This observation implies that
there exists a constant $\kappa$ so that
\begin{equation*}
|\sig\cap B_{r_3}(Q)|<\kappa r_3^{n-1}.
\end{equation*}

We will treat several cases separately. The first case is when $r_3-r_2<\min(\frac{r_1}{20},h_0)$. 
Then we can apply Theorem \ref{propagation-thm}, with $h=r_3-r_2$,  to obtain
\begin{equation*}
\|u\|_{L^2(B_{r_2}(Q))}\leq C(\|u\|_{L^2(B_{\frac{r_1}{10}}(Q'))}+\epsilon)^\delta (\|u\|_{L^2(B_{r_3}(Q))}+\epsilon)^{1-\delta},
\end{equation*}
where
\begin{equation}\label{C-i}
C=C_1\left(\frac{r_3^n}{(r_3-r_2)^{n}}\right)\left[1+\left(\frac{\kappa r_3^{n-1}}{(r_3-r_2)^{n-1}} \right)^\frac{1}{2}\right],
\end{equation}
\begin{equation}\label{delta-i}
 \delta\geq\tau^{C_2\frac{ r_3^n}{(r_3-r_2)^{n}}}.
\end{equation}

The second case is when $\frac{r_1}{10}<2h_0$, $r_3-r_2\geq\frac{r_1}{20}$. Let $r_3'=r_2+\frac{r_1}{21}$ (note $r_3'<r_3$), $h=\frac{r_1}{21}$ and again apply Theorem \ref{propagation-thm} to obtain
\begin{equation*}
\|u\|_{L^2(B_{r_2}(Q))}\leq C(\|u\|_{L^2(B_{\frac{r_1}{10}}(Q'))}+\epsilon)^\delta (\|u\|_{L^2(B_{r_3}(Q))}+\epsilon)^{1-\delta},
\end{equation*}
where
\begin{equation}\label{C-ii}
C=C_1\left(\frac{(r_2+r_1/21)^n}{(r_1/21)^{n}}\right)\left[1+\left(\frac{\kappa (r_2+r_1/21)^{n-1}}{(r_1/21)^{n-1}} \right)^\frac{1}{2}\right],
\end{equation}
\begin{equation}\label{delta-ii}
 \delta\geq\tau^{C_2\frac{ (r_2+r_1/21)^n}{(r_1/21)^{n}}}.
\end{equation}

The third and final case is when $\frac{r_1}{10}\geq 2h_0$, $r_3-r_2\geq h_0$. In this case we take $h=h_0$, and use the estimates
\begin{equation*}
|B_{r_3}(Q)|\leq (\mbox{\rm diam}(\dom))^n,\quad |B_{r_3}(Q)\cap\sig|\leq |\dom\cap\sig|.
\end{equation*}
We then have
\begin{equation*}
||u||_{L^2(B_{r_2}(Q))}\leq C(\|u\|_{L^2(B_{\frac{r_1}{10}}(Q'))}+\epsilon)^\delta (\|u\|_{L^2(B_{r_3}(Q))}+\epsilon)^{1-\delta},
\end{equation*}
where
\begin{equation}\label{C-iii}
C=C_1\frac{(\mbox{\rm diam}(\dom))^n}{h_0^{n}}\left[1+\left(\frac{ |\dom\cap\sig|}{h_0^{n-1}} \right)^\frac{1}{2}\right],
\end{equation}
\begin{equation}\label{delta-iii}
 \delta\geq\tau^{C_2\frac{ (\text{diam}(\dom))^n}{h_0^{n}}}.
\end{equation}
It follows that, in all cases, we have our three ball inequality with the constant $C$ being the maximum of the ones in \eqref{C-i}, \eqref{C-ii}, and \eqref{C-iii}, and the exponent $\delta$ being the minimum of the ones in \eqref{delta-i}, \eqref{delta-ii}, and \eqref{delta-iii}.
\end{proof}

\subsection{Proof of Theorem \ref{main-thm-1}}

Once we have established the three balls inequality in Theorem \ref{three-balls}, the proof of Theorem \ref{main-thm-1} is standard. We include it here for the benefit of the reader. Let
\begin{equation*}
r_3=\frac{h}{2},\quad r_2=\frac{1}{5}r_3=\frac{1}{10}h,\quad r_1=\frac{1}{3}r_3=\frac{1}{30}h,
\end{equation*}
and
\begin{equation*}
\tilde D=\left\{x\in\dom:  \mbox{\rm dist}(x,D)<r_1\right\},
\end{equation*}
which is an open connected subset of $\dom$, such that $D\subset \tilde D$, $\text{\rm dist}(\tilde D,\pr\dom)>h/2$. Let $y\in\tilde D$ and $\gamma\in C([0,1];\tilde D)$ be a continuous curve such that $\gamma(0)=x_0$, and $\gamma(1)=y$. Define 
\begin{equation*}
0=t_0<t_1<\cdots<t_N=1
\end{equation*}
so that
\begin{equation*}
\begin{array}{c}
t_{k+1}=\max\{t:|\gamma(t)-\gamma(t_k)|=2r_1\}, \text{ as long as } |y-\gamma(t_k)|>2r_1,\\[5pt] \text{ otherwise }N=k+1, t_N=1.
\end{array}
\end{equation*}
Then $B_{r_1}(\gamma(t_k))\cap B_{r_1}(\gamma(t_{k-1}))=\emptyset$, and $B_{r_1}(\gamma(t_{k+1}))\subset B_{r_2}(\gamma(t_{k}))$, $k=1,\ldots,N-1$. By Theorem \ref{three-balls} we have
\begin{equation*}
\|u\|_{L^2(B_{r_1}(\gamma(t_{k+1})))}+\epsilon\leq C\left(\|u\|_{L^2(B_{r_1}(\gamma(t_{k})))}+\epsilon\right)^\tau(\|u\|_{L^2(\dom)}+\epsilon)^{1-\tau},
\end{equation*}
where $k=0,\ldots,N-1$. Note that by simply modifying the constant $C$ we can add $\ep$ on both sides of \eqref{3ball}. 

Let
\begin{equation*}
m_k=\frac{\|u\|_{L^2(B_{r_0}(\gamma(t_k))}+\epsilon}{\|u\|_{L^2(\dom)}+\epsilon},
\end{equation*}
then $m_{k+1}\leq C m_k^\tau$, $k=0,\ldots,N$, and so
\begin{equation*}
m_N\leq C^{1+\tau+\cdots+\tau^{N-1}}m_0^{\tau^N}.
\end{equation*}
Since the balls $B_{r_0}(\gamma(t_k))$ are pairwise disjoint,
\begin{equation*}
N\leq \frac{|\dom|}{\omega_n r_1^n}\leq \frac{C_2|\dom|}{h^n}.
\end{equation*}
Then it is easy to see that 
\begin{equation*}
\tau^N\geq\tau^{\frac{C_2|\dom|}{h^n}},\quad C^{1+\tau+\cdots+\tau^{N-1}}\leq C^{\frac{1}{1-\tau}}.
\end{equation*}
From a family of disjoint open cubes of side $2r_1/\sqrt{n}$ whose closures cover $\R^n$, extract the finite number of cubes which intersect $D$ non-trivially: $Q_j$, $j=1,\ldots,J$. The number of these cubes satisfies $J\leq\frac{n^{n/2}|\dom|}{2^nr_1^n}$. For each $j$ there exists $w_j\in \tilde D$ such that $Q_j\subset B_{r_1}(w_j)$. Then
\begin{multline*}
\int_D |u|^2\leq \sum_{j=1}^J\int_{Q_j}|u|^2\leq
\sum_{j=1}^J\int_{B_{r_1}(w_j)}|u|^2\\\leq
JC^{2/(1-\tau)}(\|u\|_{L^2(B_{r_1}(x_0)}+\epsilon)^{2\delta}(\|u\|_{L^2(\dom)}+\epsilon)^{2(1-\delta)}.
\end{multline*}

\section{Consequences of Theorem \ref{main-thm-1}}\label{CONSEQUENCES}

In this section we list three results which are consequences of Theorem \ref{main-thm-1}. All of them are analogous to results of  \cite{ARRV}, \cite{MV}, or \cite{RS},  and  exploit the similarity of Theorem \ref{main-thm-1} to \cite[Theorem 5.1]{ARRV} (quoted above as Theorem \ref{propagation}). Since the proofs of most of these results would be identical to the ones given in  \cite{ARRV}, \cite{RS}, we will not give them here. The result analogous to that of \cite{MV} is a direct consequence of our Theorem \ref{main-thm-1} and the main result of \cite{MV}. 

Again we   assume $\dom\subset\R^n$ is an open Lipschtiz domain,  $\sig$ is a $C^2$ hypersurface with constants $r_0$, $K_0$,   $\dom\setminus \sig$ has two connected components, $\dom_\pm$, and
\begin{equation*}
\left((a^\pm_{jk})_{jk},  q  \right)\in \mathscr{V}_0(\dom_\pm,\lambda,M,K_1).
\end{equation*}

\subsection{Global propagation of smallness}

One  consequence of Theorem \ref{main-thm-1} is the following global propagation of smallness theorem.

\begin{thm}[{see \cite[Theorem 5.3]{ARRV}}]\label{main-thm-2}
Let $B_{r}(x)\subset\dom$.  If $u\in H^1(\dom)$ is a solution of 
\begin{equation*}
\el_\gamma u=f+\nabla \cdot F,\quad \|f\|_{L^2(\dom)}+\|F\|_{L^2(\dom)}\leq\epsilon,
\end{equation*}
and
\begin{equation*}
\|u\|_{L^2(B_{r}(x_0))}\leq\eta,
\end{equation*}
\begin{equation*}
\|u\|_{H^1(\dom)}\leq E,
\end{equation*}
for some $\eta>0$, $E>0$, then
\begin{equation*}
\|u\|_{L^2(\dom)}\leq (E+\epsilon)\omega\left(\frac{\eta+\epsilon}{E+\epsilon}\right),
\end{equation*}
where
\begin{equation*}
\omega(t)\leq\frac{C}{\left|\log t\right|^\mu},\quad t<1,
\end{equation*}
and $C>0$, $0<\mu<1$ depend on $\lambda$, $M$, $K_1$,  $r_0$, $K_0$, $\sig$, $\dom$, $r$.
\end{thm}

\subsection{Stability for the Cauchy problem}

Another consequence is the following stability result for the Cauchy problem for the operator $\el_\gamma$. Here  $\Gamma\subset\pr\dom$ is an open subset of the boundary.

\begin{thm}[see {\cite[Theorem 1.7]{ARRV}}]\label{T6}
Let $u\in H^1(\dom)$ be a solution of 
\begin{equation*}
\el_\gamma u=f+\nabla \cdot F,\quad \|f\|_{L^2(\dom)}+\|F\|_{L^2(\dom)}\leq\epsilon,
\end{equation*}
with $u|_{\pr\dom}\in H^{1/2}(\Gamma)$, $\sum_{jk}a_{jk}n_j\pr_k u|_{\pr\dom}\in \hbi(\Gamma)$,
\begin{equation*}\textstyle
\|u|_{\pr\dom}\|_{H^{1/2}(\sig)}+\|\sum_{jk}a_{jk}n_j\pr_k u|_{\pr\dom}\|_{\hbi(\sig)}    \leq\eta,
\end{equation*}
\begin{equation*}
\|u\|_{L^2(\dom)}\leq E_0,
\end{equation*}
for some  $\eta,\epsilon, E_0 >0$.
There exists $0<\bar h<\infty$, depending on $\lambda, L, K, \dom, \sig$ such that if
for every $0<h<\bar h$ and every open $D\subset\dom$ such that
$
\text{\rm dist} (D,\pr\dom)\geq h, 
$
we have
\begin{equation*}
\|u\|_{L^2(D)}\leq C(\epsilon+\eta)^\delta (E_0+\epsilon+\eta)^{1-\delta},
\end{equation*}
where
\begin{equation*}
C=C_1\left( \frac{|\dom|}{h^n} \right)^\frac{1}{2},\quad \delta\geq \tau^{\frac{C_2|\dom|}{h^n}},
\end{equation*}
with $C_1, C_2>0$, $\tau\in(0,1)$, depending on $\lambda$, $M$, $K_1$,  $r_0$, $K_0$, $\sig$, $\dom$, $\gam$.
\end{thm}

Finally, we state a global version of the preceding theorem.

\begin{thm}[see {\cite[Theorem 1.9]{ARRV}}]\label{T7}
Let $u\in H^1(\dom)$ be a solution of 
\begin{equation*}
\el_\gamma u=f+\nabla \cdot F,\quad \|f\|_{L^2(\dom)}+\|F\|_{L^(\dom)}\leq\epsilon,
\end{equation*}
with $u|_{\pr\dom}\in H^{1/2}(\Gamma)$, $\sum_{jk}a_{jk}n_j\pr_k u|_{\pr\dom}\in \hbi(\Gamma)$,
\begin{equation*}\textstyle
\|u|_{\pr\dom}\|_{H^{1/2}(\Gamma)}+\|\sum_{jk}a_{jk}n_j\pr_k u|_{\pr\dom}\|_{\hbi(\Gamma)}    \leq\eta,
\end{equation*}
\begin{equation*}
\|u\|_{L^2(\dom)}\leq E_0,
\end{equation*}
for some  $\eta,\epsilon, E_0 >0$.
Then
\begin{equation*}
\|u\|_{L^2(\dom)}\leq (E+\epsilon+\eta)\omega\left(\frac{\epsilon+\eta}{E+\epsilon+\eta}\right),
\end{equation*}
where
\begin{equation*}
\omega(t)\leq\frac{C}{\left|\log t\right|^\mu},\quad t<1,
\end{equation*}
and $C>0$, $0<\mu<1$ depend on $\lambda$, $M$, $K_1$,  $r_0$, $K_0$, $\sig$, $\dom$, $\gam$.
\end{thm}

\subsection{Propagation of smallness from a set of positive measure}

The next result follows easily from Theorems 1.1  of \cite{MV} and our main result. 

\begin{thm}[see {\cite[Theorem 1.1]{MV}}]
Let $u\in H^1(\dom)$ be a solution of 
$
\el_\gamma u=0.
$
Suppose $h>0$ is such that $(\dom_+)_h$ is connected, and that $E\subset(\dom_+)_h$ is a measurable set of positive measure. If
$
||u||_{L^2(E)}\leq \eta,\quad ||u||_{L^2(\dom)}\leq 1,
$
then
\begin{equation*}
||u||_{L^2(\dom_h)}\leq C |\log \eta|^{-\mu} ,
\end{equation*}
where the constants $C,\mu>0$ depend on $\lambda$, $M$, $K_1$,  $r_0$, $K_0$, $\sig$, $\dom$, $|E|$, and $h$.
\end{thm}

\begin{proof}
Note that even if $\dom_+$ is does not have Lipschitz boundary as required by \cite[Theorem 1.1]{MV}, we can still choose a slightly smaller  Lipschitz domain $\tilde\dom_+\subset \dom_+$ such that $E\subset (\tilde \dom_+)_h$. By applying \cite[Theorem 1.1]{MV}, we get that
\begin{equation*}
||u||_{L^2(B_r(x_0))}\leq C' |\log\eta|^{-\mu'},
\end{equation*}
where we have picked a ball $B_r(x_0)\subset (\tilde \dom_+)_h$. Applying our Theorem \ref{main-thm-1} with $D=\dom_h$, the result follows.
\end{proof}

\subsection{Quantitative Runge property}

The final results we would like to include are two consequences of Theorems \ref{T6} and \ref{T7}. These are  a quantitative versions of the  Runge approximation property and result that come from the work \cite{RS}. 

Let $D$, $\tilde D$   be open subsets with Lipschitz boundaries such that $D\subset\subset\tilde D\subset\subset \dom$ and define
\begin{equation*}
\s_1=\left\{ u\in H^1(D):Lu=0\text{ in }D\right\},
\end{equation*}
\begin{equation*}
\tilde\s_1=\left\{ u\in H^1(\tilde D):Lu=0\text{ in }\tilde D\right\},
\end{equation*}
\begin{equation*}
\s_2=\left\{u\in H^1(\dom):Lu=0\text{ in }\dom,u|_{\gam}\in \hb(\gam)\right\}.
\end{equation*}

The following two theorems can be proven by an argument identical to that in \cite{RS}.

\begin{thm}[see {\cite[Theorem 2]{RS}}]
There exist $\mu>0$ and $C>1$, which depend on $n$,  $\lambda$, $M$, $K_1$, $r_0$, $K_0$, $\sig$, $\dom$, $\gam$, such that for any $v\in\s_1$ and any $0<\epsilon<1$, there exists a $u\in\s_2$ such that
\begin{equation*}
\|v-u\|_{L^2(D)}\leq \epsilon\|v\|_{H^1(D)},\quad \|u|_{\pr\dom}\|_{\hb(\gam)}\leq C\exp(C\epsilon^{-\mu})\|v\|_{L^2(D)}.
\end{equation*}
\end{thm}

\begin{thm}[see {\cite[Theorem 3]{RS}}]
There exist $\mu>1$, $C>1$,  which depend on $n$, $\lambda$, $M$, $K_1$, $r_0$, $K_0$, $\sig$, $\dom$, $\gam$, such that for any $\tilde v\in\tilde\s_1$ and any $0<\epsilon<1$, there exists a $u\in\s_2$ such that
\begin{equation*}
\|\tilde v-u\|_{L^2(D)}\leq \epsilon\|\tilde v\|_{H^1(\tilde D)},\quad \|u|_{\pr\dom}\|_{\hb(\gam)}\leq C\epsilon^{-\mu}\|\tilde v\|_{L^2(D)}.
\end{equation*}
\end{thm}

\bibliography{propagation}
\bibliographystyle{plain}
\end{document}